\documentclass[12pt]{article}
\usepackage{a4wide}
\usepackage{amsmath,amssymb,amsthm}
\usepackage[mathscr]{eucal}
\usepackage{graphics}

\usepackage{epsfig}
\usepackage{psfrag}

\newcommand{\xo}{X^\omega}
\newcommand{\nucl}{\mathcal{N}}
\newcommand{\gr}{\Gamma}
\newcommand{\B}{\mathsf{B}}

\newcommand{\A}{\mathsf{A}}
\newcommand{\pol}[1][0]{{\rm Pol}({#1})}
\newcommand{\diam}{\textrm{diam}\,}

\author{Ievgen~V.~Bondarenko}
\title{\textbf{Growth of Schreier graphs of automaton groups}}

\newtheorem{theorem}{Theorem}

\newtheorem{corollary}[theorem]{Corollary}
\newtheorem{lemma}{Lemma}
\theoremstyle{definition}

\newtheorem{remark}{Remark}

\begin{document}

\maketitle

\begin{abstract}
Every automaton group naturally acts on the space $\xo$ of infinite sequences over some
alphabet $X$. For every $w\in\xo$ we consider the Schreier graph $\gr_{w}$ of the
action of the group on the orbit of $w$. We prove that for a large class of automaton
groups all Schreier graphs $\gr_w$ have subexponential growth bounded above by
$n^{(\log n)^m}$ with some constant $m$. In particular, this holds for all groups
generated by automata with polynomial activity growth (in terms of S.~Sidki),
confirming a conjecture of V.~Nekrashevych. We present applications to
$\omega$-periodic graphs and Hanoi graphs.\\

\noindent \textbf{Mathematics Subject Classification 2000}: 20F65, 05C25, 20F69

\vspace{0.1cm}\noindent \textbf{Keywords}: automaton group, Schreier graph,
subexponential growth.
\end{abstract}

\section{Introduction}

Let $G$ be a group with a generating set $S$ and acting on a set $X$. The (simplicial) Schreier graph $\gr(G,S,X)$ of the
action $(G,X)$ is the graph with the set of vertices $X$, where two vertices $x$ and $y$ are adjacent if and only if
there exists $s\in S\cup S^{-1}$ such that $s(x)=y$. Schreier graphs are generalizations of the Cayley graph of a group,
which corresponds to the action of a group on itself by the left multiplication.

The growth of Cayley graphs (growth of groups) is one of the central objects of study in geometric group theory. The
celebrated theorem of M.~Gromov characterizes groups of polynomial growth. Many classical groups (like linear groups,
solvable groups, hyperbolic groups, etc.) have either polynomial or exponential growth. The first group of intermediate
growth between polynomial and exponential was constructed by R.~Grigorchuk \cite{gri:Milnor}. Nowadays, the Grigorchuk
group is the best known example of automaton groups (also known as self-similar groups, or groups generated by automata).

In this paper, we consider the growth of Schreier graphs of automaton groups. Every automaton group $G$ generated by an
automaton $\A$ over some alphabet $X$ naturally acts on the set $X^n$ for every $n\geq 1$ and on the space $\xo$ of
right-infinite sequences over $X$. We get the associated sequence of finite Schreier graphs $\gr_n(G,\A)$ of the action
$(G,X^n)$ and the family of orbital Schreier graphs $\gr_w(G,\A)$ for $w\in\xo$ of the action of $G$ on the orbit of $w$.
The Schreier graphs $\gr_n$ and $\gr_w$ provide a large source of self-similar graphs and were studied in relation to
such topics as spectrum, growth, amenability, topology of Julia sets, etc. (see
\cite{Hecke,PhDBondarenko,gri_sunik:hanoi,gri_sunik:spectrum,amen_actions,ddmn:GraphsBasilica}).

It was noticed in \cite{Hecke} that for a few automaton groups the sequence of Schreier graphs $\gr_n$  converges to a
nice metric space with fractal structure. Further, this observation lead to the notion of a limit space of an automaton
group, and, more generally, to a limit dynamical system \cite{self_sim_groups}. The expanding property of this dynamical
system corresponds to the contracting property of a group, what brings us to the important class of contracting automaton
groups. Namely the strong contracting properties of the Grigorchuk group were used to prove that it has intermediate
growth, while in general a contracting group may have exponential growth. At the same time, all orbital Schreier graphs
$\gr_w$ of contracting groups have polynomial growth \cite{Hecke,self_sim_groups}. The degree of the growth is related to
the asymptotic characteristics of the group such as the complexity of the word problem, the Hausdorff dimension of the
limit space, the exponent of divergence of geodesics in the associated Gromov-hyperbolic self-similarity complex, etc.
(see \cite{fractal_gr_sets,PhDBondarenko}).

Most of the studied automaton groups are generated by polynomial automata. These automata were introduced by S.~Sidki in
\cite{sidki:circ}, who tried to classify automaton groups by the cyclic structure of the generating automaton. A finite
automaton is called polynomial if the simple directed cycles away from the trivial state are disjoint. The term
\textquotedblleft polynomial" comes from the equivalent definition, where a finite automaton is polynomial if the number
of paths of length $n$ avoiding the trivial state in the automaton  grows polynomially in~$n$.

It is an open problem whether contracting groups and groups generated by polynomial automata are amenable. However, it is
known that these groups do not contain non-abelian free subgroups \cite{free_subgroups,sidki:poly} (but may contain free
semigroups and be of exponential growth), and that groups generated by polynomial automata of degree $0$ and $1$ (bounded
and linear automata) are amenable~\cite{amenability,linear}. In \cite{free_subgroups} V.~Nekrashevych introduced a
general approach to the existence of free subgroups in automaton groups and applied it to contracting groups and to
groups generated by polynomial automata. {It was shown that there exists certain trichotomy for groups acting on rooted
trees that involves the absence of non-abelian free subgroups as one of the three cases. In order to eliminate one of
these cases for groups generated by polynomial automata, V.~Nekrashevych proved that their orbital Schreier graphs
$\gr_w$ are amenable and conjectured \cite[page~858]{free_subgroups} that these graphs have subexponential growth. This
conjecture was based on the results from \cite{SchreierInter,omega_periodic}, where it is shown that for the group
generated by one of the simplest polynomial automata all Schreier graphs $\gr_w$ have intermediate growth. The main goal
of this paper is to prove this conjecture.

\begin{theorem}\label{theorem_main}
Let $G$ be a group generated by a polynomial automaton of degree $m$. There exists a constant $A$ such that all orbital
Schreier graphs $\gr_w(G)$ for $w\in\xo$ have subexponential growth not greater than $A^{(\log n)^{m+1}}$.
\end{theorem}

In order to prove Theorem~1, we establish certain weak contracting properties of groups generated by polynomial automata.
In particular, we prove that the word problem in these groups is solvable in subexponential time bounded above by
$B^{(\log n)^{m+1}}$ for some constant~$B$. We also prove Theorem~\ref{theorem_main} for a more general class of automata
(see Theorem~\ref{theorem_growth_Schr_general} in Section~\ref{Generalized Version}), which generalizes both polynomial
automata and contracting groups.

In Section~\ref{section_Examples} we apply Theorem~\ref{theorem_main} to construct automaton groups with Schreier graphs
$\gr_w$ of intermediate growth. In the first example we consider a class of automaton groups, whose Schreier graphs
$\gr_w$ are a generalized version of the $\omega$-periodic graph of intermediate growth studied in
\cite{omega_periodic,SchreierInter}. This example shows that the upper bound in Theorem~\ref{theorem_main} is
asymptotically optimal by providing an example of a polynomial automaton of degree $m$ for each $m\in\mathbb{N}$, whose
all orbital Schreier graphs $\gr_w$ for $w\in X^{\omega}$ have growth not less than $B^{(\log n)^{m+1}}$ for some
constant $B>1$. Another example comes from the well-known Hanoi Tower Game on $k$ pegs. This game was modeled by
automaton groups $G_{(k)}$ in \cite{gri_sunik:hanoi}, and, using known estimates on the complexity of the Hanoi Tower
Game, it was noticed that the orbital Schreier graph $\gr_{0^{\infty}}(G_{(k)})$ for $k\geq 4$ has intermediate growth.
The automata generating groups $G_{(k)}$ for $k\geq 4$ are not polynomial, but we apply similar arguments to prove that
all orbital Schreier graphs $\gr_w(G_{(k)})$ have intermediate growth.

\section{Automaton groups and their Schreier graphs}

In this section we recall all necessary definitions; see \cite{self_sim_groups} for a more detailed introduction.

\vspace{0.2cm}\textbf{Spaces of words.} Let $X$ be a finite alphabet with at least two letters. Denote by
$X^{*}=\{x_1x_2\ldots x_n | x_i\in X, n\geq 0\}$ the set of all finite words over $X$ (including the empty word denoted
$\emptyset$). The length of a word $v=x_1x_2\ldots x_n\in X^n$ is denoted by $|v| = n$. Let $X^{\omega}$ be the space of
all infinite sequences (words) $x_1x_2\ldots$, $x_i\in X$, with the product topology of discrete sets $X$. The space
$\xo$ is homeomorphic to the Cantor set, i.e., it is a compact totally disconnected metrizable topological space without
isolated points.

\vspace{0.2cm}\textbf{Automata.} An invertible \textit{automaton} $\A$ over the alphabet $X$ is the triple
$(S,\pi,\lambda)$, where $S$ is the set of states of the automaton, $\lambda:S\times X\rightarrow S$ is the transition
function, and $\pi:S\times X\rightarrow X$ is the output function such that for every $s\in S$ the map
$\pi(s,\cdot):X\rightarrow X$ is a permutation on $X$. All automata in the paper are invertible, and further we will omit
the term invertible. An automaton is \textit{finite} if the set of its states is finite. An automaton $\A$ is represented
by a directed labeled graph (Moore diagram), whose vertices are the states of $\A$ and for every state $s\in S$ and every
letter $x\in X$ there is an arrow $s\rightarrow \lambda(s,x)$ labeled by the pair $x|\pi(s,x)$. This diagram contains
complete information about the automaton, and we identify the automaton with its Moore diagram. The notation $\A$ is also
used for the state set of the automaton $\A$, so that one can talk about a state $s\in\A$. A subset $\B\subset\A$ (with
induced edges) is a \textit{subautomaton} of $\A$ if $\lambda(s,x)\in\B$ for every $s\in\B$ and $x\in X$.

\vspace{0.2cm}\textbf{Automaton groups.} For every state $s\in\A$ and every finite word $v=x_1x_2\ldots x_n\in X^{*}$
there exists a unique path in the automaton $\A$ starting at the state $s$ and labeled by $x_1|y_1$, $x_2|y_2$, \ldots,
$x_n|y_n$ for some $y_i\in X$. Then the word $y_1y_2\ldots y_n$ is called the image of $x_1x_2\ldots x_n$ under $s$, and
the end state of this path is denoted by $s|_v$. In other words, for every finite word $v\in X^{*}$ we define the
\textit{image $s(v)$ of $v$ under $s$} and the \textit{state} $s|_v$ recursively by the rules
\begin{align*}
s|_x=\lambda(s,x), \quad  s|_{xv}=s|_x|_v, \quad \mbox{ and } \quad s(x)=\pi(s,x), \quad  s(xv)=s(x)s|_x(v)
\end{align*}
for every $x\in X$ and $v\in X^{*}$. The action of states of $\A$ on the set $X^n$ can be given by the automaton
$\A^{(n)}$ obtained from $\A$ by passing to the power $X^n$ of the alphabet $X$. The automaton $\A^{(n)}$ is defined over
the alphabet $X^n$ and has the same states as $\A$, but the arrows are $s\rightarrow s|_v$ labeled by $v|s(v)$ for every
$s\in\A$ and $v\in X^n$.

In the same way we get an action of every state $s\in\A$ on the space $\xo$ by looking at the infinite paths in the
automaton. Since the automata are invertible, all transformations $s\in\A$ are also invertible. The group $G$ generated
by all states $s\in\A$ is called the \textit{automaton group} generated by $\A$. Since every state preserves the length
of words in its action on $X^{*}$, the states act by isometries on the space $\xo$, and every automaton group is a
subgroup of $\textrm{Iso}(\xo)$. The automaton groups generated by $\A$ and $\A^{(n)}$ coincide viewed as subgroups of
$\textrm{Iso}(\xo)$ with the natural identification $\xo=(X^n)^{\omega}$. All automata in the paper are supposed to be
minimized, i.e., different states of a given automaton act differently on~$\xo$. Hence we identify the states with the
respective transformations of~$\xo$.

An alternative approach is through self-similar actions. A faithful action of a group $G$ on the set $X^{*}\cup\xo$ is
called \textit{self-similar} if for every $g\in G$ and $v\in X^{*}$ there exist $u\in X^{|v|}$ and $h\in G$ such that
$g(vw)=uh(w)$ for all $w\in X^{*}\cup\xo$. The element $h$ is called the \textit{restriction (state) of $g$ at $v$} and
is denoted by $g|_v$. We are using left actions, i.e., $(g_1g_2)(v)=g_1(g_2(v))$, and hence the restrictions have the
property
\[
(g_1g_2)|_v=g_1|_{g_2(v)}g_2|_v\quad \mbox{ for any } g_1,g_2\in G \mbox{ and } v\in X^{*}.
\]
The \textit{complete automaton} $\A(G)$ of a self-similar action of the group $G$ is an automaton with the set of states
$G$ and with the arrows $g\rightarrow g|_x$ labeled by $x|g(x)$ for every $g\in G$ and $x\in X$.

An automaton group $G$ is called \textit{contracting} if there exists a finite subset $\nucl\subset G$ with the property
that for every $g\in G$ there exists $n\in\mathbb{N}$ such that $g|_v\in\nucl$ for all words $v\in X^{*}$ of length
$|v|\geq n$. The smallest set $\nucl$ with this property is called the \textit{nucleus} of the group. Notice that a
finitely generated contracting group can be generated by a finite automaton. A group $G$ generated by a finite automaton
is contracting if and only if the complete automaton $\A(G)$ of the group contains only finitely many simple directed
cycles. The nucleus is the union of these cycles together with all elements that can be reached following directed paths
from the cycles.

Throughout the paper by a \textit{cycle} in an automaton we mean a simple directed cycle. States that lie on cycles are
called \textit{circuit} states. An element $g$ of an automaton group $G$ is circuit if there exists a nonempty word $v\in
X^{*}$ such that $g|_v=g$.

\vspace{0.2cm}\textbf{Polynomial automata.} A cycle in an automaton is called trivial if it is a loop at the state acting
trivially on $\xo$ (trivial state, denoted $e$). A finite automaton $\A$ is called \textit{polynomial} if different
nontrivial cycles in $\A$ are disjoint. A polynomial automaton $\A$ is \textit{of degree $m$} if the largest number of
nontrivial cycles in $\A$ connected by a directed path is equal to $m+1$. A finite automaton $\A$ is polynomial (of
degree $\leq m$) if and only if the number of directed paths in $\A$ of length $n$ that do not pass through the trivial
state is bounded by a polynomial in $n$ (of degree $\leq m$). The set of all states (transformations of $\xo$) of all
polynomial automata of degree $\leq m$ forms a group $\pol[m]$ called the \textit{group of polynomial automata of degree}
$m$. Every finitely generated subgroup of $\pol[m]$ is a subgroup of some automaton group generated by a polynomial
automaton.

We classify the states of a polynomial automaton $\A$ as follows. A state $s\in\A$ is \textit{finitary} if there exists
$n\in\mathbb{N}$ such that $s|_v=e$ for all $v\in X^n$. The finitary states are precisely the elements of $\pol[-1]$. The
states from $\pol[m]\setminus\pol[m-1]$ are called \textit{polynomial of degree~$m$}. For every polynomial automaton $\A$
there exists $n$ such that for every $s\in\A$ and $v\in X^n$ the state $s|_v$ is either circuit or has degree less than
the degree of $s$. Notice that if $s\in\A$ is a nontrivial circuit state then for every $n$ there exists precisely one
word $v\in X^n$ such that $s$ and $s|_v$ have the same degree as polynomial states, and $s|_u$ for all $u\in X^n$, $u\neq
v$, has degree less than the degree of $s$.

\vspace{0.2cm}\textbf{Schreier graphs.} Let $G$ be an automaton group generated by a finite subset $S$. The
\textit{Schreier graph $\gr_n(S)=\gr_n(G,S)$} is the graph with the set of vertices $X^n$, where two vertices $v$ and $u$
are adjacent if there exists $s\in S\cup S^{-1}$ such that $s(v)=u$. The \textit{orbital Schreier graph
$\gr_w(S)=\gr_w(G,S)$} for $w\in\xo$ is the graph, whose vertex set is the orbit $G(w)$ and two vertices $v$ and $u$ are
adjacent if there exists $s\in S\cup S^{-1}$ such that $s(v)=u$. The orbital Schreier graphs $\gr_w(S)$ are precisely the
connected components of the Schreier graph $\gr(G,S,\xo)$ of the action $(G,\xo)$. Every orbital Schreier graph $\gr_w$
is a limit of the finite Schreier graphs $\gr_n$ in the local Gromov-Hausdorff topology on pointed graphs.

\vspace{0.2cm}\textbf{Growth of graphs.} Let $\gr$ be a locally finite connected graph. The \textit{growth function}
$\gamma_v(n)$ of $\gr$ with respect to its vertex $v$ is equal to the number of vertices in the closed ball $B(v,n)$ of
radius $n$ centered at $v$. There is a partial order on the growth functions. Given two functions
$f,g:\mathbb{N}\rightarrow\mathbb{N}$ we say that $f$ has \textit{growth} not greater than $g$ (denoted $f\prec g$) if
there exists a constant $C>0$ such that $f(n)\leq g(Cn)$ for all $n\in\mathbb{N}$. If $f\prec g$ and $g\prec f$ then $f$
and $g$ are called equivalent $f\sim g$ and have the same growth. Formally, by the growth we can understand the
equivalence class of a function. Then, for any two vertices of the graph $\gr$, the respective growth functions are
equivalent, and one can talk about the growth of $\gr$.

A graph $\gr$ has \textit{subexponential growth} if its growth function has growth not greater than the exponential
growth $a^n$ with $a>1$ and is not equivalent to it. The growth is \textit{superpolynomial} if it is greater than every
polynomial function, and the growth is \textit{intermediate} if it is superpolynomial and subexponential.

Let $G$ be a finitely generated group with finite generating set $S$. The \textit{length} $l(g)$ of $g\in G$ with respect
to $S$ is equal to the distance between the trivial element $e$ and $g$ in the Cayley graph $\gr(G,S)$, i.e., $l(e)=0$
and
\[
l(g)=l_S(g)=\min\{ n\, |\, g=s_1s_2\ldots s_n \mbox{ for } s_i\in S\cup S^{-1}\}.
\]
The growth function $\gamma$ of the group $G$ is the growth function of the Cayley graph $\gr(G,S)$ with respect to the
vertex $e$, i.e., $\gamma(n)$ is equal to the number of elements $g\in G$ of length $l(g)\leq n$. The growth functions of
the Cayley graphs $\gr(G,S_1)$ and $\gr(G,S_2)$ for any two finite generating sets $S_1$ and $S_2$ are equivalent.

The growth of Schreier graphs $\gr_w(G,S)$ of an automaton group $G$ also does not depend on the choice of a finite
generating set $S$ of the group. Working with automaton groups it is useful to assume that a generating set $S$ is
\textit{self-similar} (automaton), i.e., $s|_v\in S$ for every $s\in S$ and $v\in X^{*}$. For example, we will frequently
use the following observation. If there is a presentation $g=s_1s_2\ldots s_n$ for $s_i\in S$ then for every $v\in X^{*}$
we get the \textit{induced presentation}
\begin{equation}\label{eqn_induced_present}
g|_v=s_1|_{v_1}s_2|_{v_2}\ldots s_{n-1}|_{v_{n-1}} s_n|_{v_n},
\end{equation}
where $v_n=v$ and $v_{i-1}=s_i(v_i)$ for $i=n,n-1,\ldots, 2$. In particular, if $S$ is self-similar then $s_i|_{v_i}\in
S$ and  $l(g|_v)\leq l(g)$. Also, it is usually assumed that $S$ is \textit{symmetric}, i.e., $S=S^{-1}$.

All logarithms in the paper are with base $2$, except if directly indicated. Usually, a logarithm appears as $C\log n$,
and the base can be hidden in the constant $C$. Also (to avoid some multiple brackets) we use the convention that $\log 0
=1$ and $\log 1 =1$ so that $\log n>0$ for all $n\in\mathbb{N}\cup \{0\}$, otherwise one can just replace $\log n$ by
$\log(n+2)$.

\section{Proof of Theorem~1}\label{Proof}

Let us recall how to prove that for contracting groups the Schreier graphs $\gr_w$ have polynomial growth (see
\cite[Section~2.13.4]{self_sim_groups}, \cite{Hecke}). Let $G$ be a finitely generated contracting group with nucleus
$\nucl$, and let $S$ be a finite symmetric self-similar generating set of $G$ that contains $\nucl$. We can choose a
constant $C$ such that $(s_1s_2)|_v\in\nucl$ for all $s_1,s_2\in S$ and every word $v\in X^{*}$ of length $|v|\geq C$.
Consider an element $g\in G$ and let $g=s_1s_2\ldots s_n$ for $s_i\in S$ with $n=l(g)$. Then
$(s_is_{i+1})|_v\in\nucl\subset S$ and hence the element $g|_v$ has length $\leq (n+1)/2$. It follows that $g|_v\in\nucl$
for all words $v\in X^{*}$ of length $|v|\geq C\log n$. This justifies the term contracting group: the length of
restrictions exponentially decreases (\textquotedblleft contracts") until they become elements of the nucleus.

Now consider the Schreier graph $\gr_w=\gr_w(G,S)$ for a sequence $w=x_1x_2\ldots\in\xo$. Let $B(w,n)$ be the ball of
radius $n$ in the graph $\gr_w$ centered at the vertex $w$. Notice that if $g(v_1w_1)=v_2w_2$ for $v_1,v_2\in X^k$ and
$w_1,w_2\in\xo$ then $g|_{v_1}(w_1)=w_2$. Let $k$ be the least integer greater than $C\log n$. Then each sequence in the
ball $B(w,n)$ is of the form $v_2w_2$ for some $v_2\in X^k$ and $w_2=h(x_{k+1}x_{k+2}\ldots)$ for some $h\in \nucl$.
Hence
\[
|B(w,n)|\leq |X|^{C\log n +1}\cdot |\nucl(x_{k+1}x_{k+2}\ldots)|\leq |X|^{C\log n +1}\cdot |\nucl|,
\]
which is a polynomial in $n$. More precise estimate can be given using the contracting coefficient of the group (see
\cite[Proposition~2.13.8]{self_sim_groups}).

We want to apply similar arguments to groups generated by polynomial automata, and first we will establish certain weak
contracting properties of these groups.

\begin{lemma}\label{lemma_weak_contr_poly}
Let $G$ be a finitely generated subgroup of $\pol[m]$. There exists a constant $C$ such that for every $g\in G$ and every
word $v\in X^{*}$ of length $|v|\geq C\left(\log l(g)\right)^{m+1}$ the state $g|_v$ is either circuit or belongs to
$\pol[m-1]$.
\end{lemma}

To prove Lemma~\ref{lemma_weak_contr_poly} we use induction on $m$. However, there is one slight difficulty that in order
to apply the induction hypothesis we may need to consider the length of the element $g|_v$ for a different generating set
than the length of $g$. To overcome this problem we will make a few assumptions so that the groups involved in the
induction have consistent generating sets.

Without loss of generality, we can assume that the group $G$ is an automaton subgroup of $\pol[m]$. Let $S$ be a finite
symmetric self-similar generating set of $G$. Let $G_k$ be an automaton subgroup of $G$ generated by $S_k=S\cap\pol[k]$
for $k=-1,0,\ldots, m$. The generating set $S_k$ is again symmetric and self-similar. Notice that, in general, the group
$G_k$ may not coincide with $G\cap\pol[k]$.

Let us prove that the length of cycles in the complete automaton $\A(G)$ of the group $G$ is bounded. Without loss of
generality, by passing to a power of the alphabet, we can assume that the alphabet $X$ and the generating set $S$ satisfy
the following conditions:\\
\indent 1) for every $s\in S$ and $x\in X$ the state $s|_x$ is either circuit or has degree \\
\indent \indent  less than the degree of $s$;\\
\indent 2) for every nontrivial circuit state $s\in S$ there exists (unique) $x\in X$ such  \\
\indent\indent that $s|_x=s$ (every cycle in $S$ is a loop).\\
(for example, one can pass to the alphabet $X^k$, where $k$ is a multiple of the length of every cycle in $S$ and is
greater than the diameter of $S$). Let us prove that then every cycle in the automaton $\A(G)$ is a loop, i.e., for every
circuit element $g\in G$ there exists a letter $x\in X$ such that $g|_x=g$. Assume $g|_v=g$ for a nonempty word $v\in
X^{*}$ and let $g=s_1s_2\ldots s_n$ for $s_i\in S$ with $n=l(g)$. Notice that $s_i\not\in S_{-1}$ for every $i$, because
otherwise $g|_v=g$ can be expressed as a product $s_1s_2\ldots s_l$ with $l<n$. Among all presentations of $g$ as a
product $s_1s_2\ldots s_n$ with $n=l(g)$ we choose presentations with the maximal number of circuit generators from
$S_0$, and among such presentations we consider those with the maximal number of the rest of the generators from $S_0$.
Further, among selected presentations we choose those with the maximal number of circuit generators from $S_1$, and then
with the maximal number of the rest of the elements from $S_1$, and so on for $S_2, S_3, \ldots, S_m$. We obtain some
specific presentations of the element $g$. Let $g=s_1s_2\ldots s_n$ be one of such presentations and consider the induced
presentation \eqref{eqn_induced_present}:
\[
g=g|_v=s_1|_{v_1}\, s_2|_{v_2}\, \ldots \, s_n|_{v_n},
\]
where $v_{i-1}=s_i(v_i)$ with $v_n=v$. By our choice of the presentation $g=s_1s_2\ldots s_n$ we should get
\[
s_1|_{v_1}=s_1, \quad s_2|_{v_2}=s_2, \quad \ldots, \quad s_n|_{v_n}=s_n.
\]
Then
\[
s_1|_{x_1}=s_1,\quad s_2|_{x_2}=s_2,\quad \ldots,\quad s_n|_{x_n}=s_n,
\]
where $x_i$ is the first letter of $v_i$, and we get $g|_{x_n}=g$. It also follows that $g|_y\in G_{m-1}$ for $y\in X$,
$y\neq x_n$, and in its induced presentation $g|_y=s_1|_{y_1}\, s_2|_{y_2}\, \ldots \, s_n|_{y_n}$ every generator
$s_i|_{y_i}$ belongs to the set $S_{m-1}$. In particular, the length of $g|_y$ with respect to the generating set
$S_{m-1}$ of the group $G_{m-1}$ is not greater than $n$. Further we will use this property without specifying the
generating set.

To the rest of this section we always assume that the group $G$ and the generating set $S$ satisfy the above assumptions.

We will prove a slightly stronger formulation of Lemma~1.

\vspace{0.2cm} \noindent \textbf{Lemma~1*.} \textit{There exists a constant $C$ such that for every $g=s_1s_2\ldots
s_n\in G$ for $s_i\in S$ and every word $v\in X^{*}$ of length $|v|\geq C\left(\log n\right)^{m+1}$ either $g|_v\in
G_{m-1}$ and in the induced presentation $g|_v=s_1|_{v_1}s_2|_{v_2}\ldots s_n|_{v_n}$ every generator $s_i|_{v_i}$
belongs to $S_{m-1}$, or the element $g|_v$ is circuit, $g|_v|_x=g|_v$ for some $x\in X$, and the induced presentations
of $g|_v$ and $g|_v|_x$ coincide as words in generators. }\vspace{0.2cm}

\begin{proof}
The proof goes by induction on $m$. For $m=-1$ the group $G$ is a subgroup of the group $\pol[-1]$ of finitary elements.
There exists a constant $C$ such that for all words $v$ of length $|v|\geq C$ we have $s|_v=e$ for every $s\in S$
(actually by our assumption 1) on the generating set $S$ we can take $C=1$). Hence for every product $g=s_1s_2\ldots
s_n\in G$ for $s_i\in S$ we get $g|_v=e$ and in the induced presentation of $g|_v$ every element is trivial.

We assume the lemma holds for the groups $G_{k}$ (with $m$ replaced by $k$) for $k=-1,0,\ldots, m-1$ with some common
constant $C_1$.

Let $g\in G$ be a circuit element and $g|_x=g$ for a letter $x\in X$. Let $g=s_1s_2\ldots s_k$, $s_i\in S$, be a
presentation such that the induced presentation of $g=g|_x$ coincides with $s_1s_2\ldots s_k$ as a word in generators.
Then for every word $v\in X^{*}$ either $g|_v\in G_{m-1}$ and every generator in the induced presentation of $g|_v$
belong to $S_{m-1}$, or $g|_v=g$ (in the case $v=xx\ldots x$) and the induced presentation of $g|_v$ coincides with
$s_1s_2\ldots s_k$. Consider the product $hg$ for $h=t_1t_2\ldots t_n\in G_{m-1}$, $t_i\in S_{m-1}$, and its restrictions
$(hg)|_v=h|_{g(v)} g|_{v}$ for words $v\in X^{*}$ of length $|v|\geq C_1\left(\log n\right)^m$. We have the following
cases. If $g|_{v}\in G_{m-1}$ then $(hg)|_v\in G_{m-1}$, so further assume $g|_{v}=g$ and let us look at $h|_{g(v)}$. By
the assumption on the length of the word $v$ the element $h|_{g(v)}$ either belongs to $G_{m-2}$ or it is
circuit (with the specific induced presentation).\\
\indent If $h|_{g(v)}\in G_{m-2}$ then $(hg)|_v=h'g$ for $h'\in G_{m-2}$.\\
\indent If $h|_{g(v)}$ is circuit and $h|_{g(v)}|_{g(x)}=h|_{g(v)}$ then $(hg)|_v$ is circuit, here $(hg)|_v|_{x}=(hg)|_v$.\\
\indent If $h|_{g(v)}$ is circuit and $h|_{g(v)}|_{g(x)}\neq h|_{g(v)}$ then  $(hg)|_{vx}=h'g$ for $h'\in G_{m-2}$.\\
Hence in all cases the element $(hg)|_{u}$ for $u\in X^{|v|+1}$ is either circuit, or belongs to $G_{m-1}$, or it is of
the form $h'g$ for $h'\in G_{m-2}$. In the last case we can apply the same arguments to the product $h'g$. After $m$
steps we get either a circuit state or an element of $G_{m-1}$. In the worst case we need to take the words of length
\[
C_1(\log n)^m+1 + C_1(\log n)^{m-1}+1+\ldots +C_1(\log n)^{0}+1.
\]
Choose a constant $C_2$ such that $C_2(\log n)^m$ is greater than the value of the equation above. Then the element
$(hg)|_v$ for words $v\in X^{*}$ of length $|v|\geq C_2(\log n)^m$ is either circuit or belongs to $G_{m-1}$, and the
conditions on the induced presentation of $(hg)|_v$ are satisfied.

Let $g_1,g_2\in G$ be circuit elements and consider the product $g_1hg_2$ for $h\in G_{m-1}$. Assume $g_i$ is expressed
as a product in generators of length $k_i$ with the same properties as the product of $g$ above, and $h$ is expressed as
a product of $n$ generators from $S_{m-1}$. Then the element $(g_1hg_2)|_v$ for $v\in X^{*}$ of length $|v|\geq C_3(\log
(k_1+n+k_2))^m$ with $C_3=2C_2$ is either circuit or belongs to $G_{m-1}$ with the specific induced presentation.

Now consider an arbitrary element $g=s_1s_2\ldots s_n\in G$ for $s_i\in S$. Let us partition the product $s_1s_2\ldots
s_n$ on blocks
\begin{equation}\label{eqn_lemma_g=ghgh...}
g=h_0g_1h_1g_2h_2\ldots g_lh_l
\end{equation}
such that every block $h_i$ contains only generators from $S_{m-1}$ (may be trivial), and $g_i$ is either a generator
from $S\setminus S_{m-1}$, or $g_i$ is circuit and the induced presentation of $g_i|_x$ coincides with the presentation
of $g_i$ for some $x\in X$. Notice that the sum of lengths of all $h_i$ and $g_i$ is equal to $n$. By passing to the
induced presentation of the state $g|_x$ on any letter $x\in X$ we can assume that every element $g_i$ is circuit.
Consider every product $g_1h_1g_2$, $g_3h_3g_4$, and so on, in the presentation (\ref{eqn_lemma_g=ghgh...}). Every such
block restricted to a word of length $\geq C_3(\log n)^m$ is either circuit or belongs to $G_{m-1}$. Hence the state
$g|_v$ for words $v\in X^{*}$ of length $|v|\geq C_3(\log n)^m+1$ can be expressed as a product
(\ref{eqn_lemma_g=ghgh...}) with $\leq (l+1)/2$ positions with some circuit elements $g_i$. Applying the same procedure
$\log l+1$ times we get either a circuit state or an element of $G_{m-1}$. Choose a constant $C$ such that $C(\log
n)^{m+1}\geq (C_3(\log n)^m+1)(\log n +1)$ (here we use $l\leq n$). Then $g|_v$ for $v\in X^{*}$ of length $|v|\geq
C(\log n)^{m+1}$ is either circuit or belongs to $G_{m-1}$, and in both cases it has the required induced presentation.
\end{proof}

\begin{corollary}
The word problem in every finitely generated subgroup of $\pol[m]$ is solvable in subexponential time.
\end{corollary}
\begin{proof}
We use the same notations and assumptions as above.

Consider an element $g=s_1s_2\ldots s_n\in G$ for $s_i\in S$. Let $k$ be the least integer greater than $C(\log
n)^{m+1}$. The element $g$ is trivial if and only if it acts trivially on $X^k$ and every element $g|_v$ for $v\in X^k$
is trivial. Notice that the size of $X^k$ is subexponential in $n$. Consider the induced presentation
$g|_v=s'_1s'_2\ldots s'_n$ with $s'_i=s_i|_{v_i}$. If $s'_i\in S_{m-1}$ for all $i$, then the problem reduces to the word
problem in the group $G_{m-1}$. Otherwise, $g|_v$ is circuit and there exists $x\in X$ such that the induced presentation
$g|_v|_x=s'_1|_{x_1}s'_2|_{x_2}\ldots s'_n|_{x_n}$ coincides with the presentation $g|_v=s'_1s'_2\ldots s'_n$ letter by
letter, i.e., $s'_i|_{x_i}=s'_i$ for all~$i$. Then $g|_v$ is trivial if and only if it acts trivially on $X$ and every
element $g|_v|_y\in G_{m-1}$ for $y\in X$, $y\neq x$, is trivial. Again the problem reduces to the word problem in
$G_{m-1}$. By induction we conclude that the word problem is solvable in subexponential time with an upper bound
$|X|^{C_1(\log n)^{m+1}}$ for some constant $C_1$.
\end{proof}

We are ready to prove the main result.

\begin{theorem}\label{theorem_growth_Schreier_poly}
Let $G$ be a finitely generated subgroup of $\pol[m]$. There exists a constant $C$ such that every orbital Schreier graph
$\gr_w(G)$ for $w\in\xo$ has subexponential growth not greater than $|X|^{C\left(\log n\right)^{m+1}}$.
\end{theorem}
\begin{proof}
The proof goes by induction on $m$. For $m=-1$ the group $G<\pol[-1]$ is finite, and every Schreier graph $\gr_w$ has at
most $|G|$ vertices. We can also start the induction from $m=0$ using the fact that in this case the group $G<\pol[0]$ is
contracting (see \cite{bn:pcf}).

We suppose by induction that the statement holds for the group $G_{m-1}$ and every Schreier graph
$\gr_w(G_{m-1},S_{m-1})$ has subexponential growth not greater than $|X|^{C_1\left(\log n\right)^{m}}$ with some constant
$C_1$.

Fix a sequence $w=x_1x_2\ldots\in\xo$ and consider the ball $B(w,n)$ in the graph $\gr_w=\gr_w(G,S)$ of radius $n$
centered at the vertex $w$. If $g(v_1w_1)=v_2w_2$ for $v_1,v_2\in X^k$ and $w_1,w_2\in\xo$ then $g|_{v_1}(w_1)=w_2$.
Hence for every fixed $k$ each sequence in the ball $B(w,n)$ is of the form $v_2w_2$ for some $v_2\in X^k$ and
$w_2=h(x_{k+1}x_{k+2}\ldots)$ for some $h\in \nucl(n,k)$, where
\[
\nucl_{(n,k)}=\{g|_{x_1x_2\ldots x_k} : g\in G \mbox{ and } l(g)\leq n \}.
\]
It follows that
\begin{equation}\label{eqn_theorem_growth_estimate}
|B(w,n)|\leq |X|^k\cdot |\nucl_{(n,k)}(v)|,
\end{equation}
where $v=x_{k+1}x_{k+2}\ldots$.

Let $H_n=\{h\in G_{m-1}: l_{S_{m-1}}(h)\leq n\}$ be the ball of radius $n$ in the group $G_{m-1}$ with respect to its
generating set $S_{m-1}$.

Let $k$ be the least integer greater than $C(\log n)^{m+1}$ given in Lemma~\ref{lemma_weak_contr_poly}.  Then for every
$g\in\nucl_{(n,k)}\setminus H_n$ there exists $x\in X$ such that $g|_x=g$. Hence
\[
\nucl_{(n,k)}\subset \bigcup_{x\in X} \nucl^x_{(n,k)} \cup H_n,
\]
where $\nucl^x_{(n,k)}=\{g\in\nucl_{(n,k)} : g|_x=g\}$ for $x\in X$. By induction hypothesis the size of the orbit
$H_n(v)$ is not greater than $|X|^{C_1(\log n)^m}$. Let us estimate the size of the orbits $\nucl_{(n,k)}^x(v)$. Let
$x\in X$ be the first letter of the word $v$ and consider the following cases. If $v=xx\ldots=x^{\infty}$ then for
$g\in\nucl_{(n,k)}^x$ we have $g(v)=z^{\infty}$ for some $z\in X$; and for $g\in\nucl_{(n,k)}^y$ with $y\in X$, $y\neq
x$, we have $g(v)=zh(x^{\infty})$ for some $z\in X$ and $h\in H_n$. Hence we get estimates
\[
|\nucl_{(n,k)}^x(v)|\leq |X|\quad \mbox{ and }\quad |\nucl_{(n,k)}^y(v)|\leq |X|\cdot |H_n(v)|
\]
for $y\in X$, $y\neq x$. If $v=x^lx_1v_1$ for $x_1\neq x$ then for $g\in\nucl_{(n,k)}^x$ we have $g(v)=z^lz_1h(v_1)$ for
some $z,z_1\in X$ and $h\in H_n$; and for $g\in\nucl_{(n,k)}^y$ with $y\in X$, $y\neq x$, we have
$g(v)=zh(x^{l-1}x_1v_1)$ for some $z\in X$ and $h\in H_n$. Hence we get estimates
\[
|\nucl_{(n,k)}^x(v)|\leq |X|^2\cdot |H_n(v_1)|\quad \mbox{ and }\quad |\nucl_{(n,k)}^y(v)|\leq |X|\cdot
|H_n(x^{l-1}x_1v_1)|
\]
for $y\in X$, $y\neq x$. Summarizing all estimates we get
\[
|B(w,n)|\leq |X|^{C(\log n)^{m+1}+1}\cdot |X|^3\cdot |X|^{C_1(\log n)^m}\leq |X|^{C_2(\log n)^{m+1}}
\]
for some constant $C_2$.
\end{proof}

As a corollary we recover the following result from \cite[Corollary~4.6]{free_subgroups}.

\begin{corollary}
The orbital Schreier graphs $\gr_w$ for $w\in\xo$ of groups generated by polynomial automata are amenable.
\end{corollary}

\begin{remark}
It is shown in \cite{linear} that every group generated by a polynomial automaton of
degree $m$ embeds in a certain ``mother'' group, which is generated by polynomial
automata similar to the one in the first example of Section~\ref{section_Examples}.
Hence, one can establish Theorem~\ref{theorem_growth_Schreier_poly} just by estimating
the growth of Schreier graphs of the mother groups. However, the fundamental steps of
the proof remain the same. Also, for a particular automaton
Lemma~\ref{lemma_weak_contr_poly}* may hold for a smaller value of $m$ than the degree
of the automaton, and we may get a better estimate on the growth of Schreier graphs.
\end{remark}

\section{A generalization of Theorem~\ref{theorem_main}}\label{Generalized Version}

Theorem~\ref{theorem_main} can be generalized to automaton groups with a certain combination of contracting and
polynomial properties.

\begin{theorem}\label{theorem_growth_Schr_general}
Let $\A$ be a finite automaton with subautomaton $\B$ such that different cycles in $\A\setminus \B$ are disjoint, and
the group generated by $\B$ is contracting. Let $G$ be the automaton group generated by $\A$. There exists a constant $C$
such that every orbital Schreier graph $\gr_w(G)$ for $w\in\xo$ has subexponential growth bounded above by
$|X|^{C\left(\log n\right)^{m+1}}$, where $m$ is the maximal number of different cycles in $\A\setminus \B$ connected by
a directed path.
\end{theorem}
\begin{proof}
The proof goes by induction on $m$ following the same strategy as in the proof of
Theorem~\ref{theorem_growth_Schreier_poly}.

First, we make a few assumptions about the generating sets to make our life easier. We can assume that the generating
sets $\A$ and $\B$ are symmetric, and $\B$ contains the nucleus $\nucl$ of the contracting group $\langle \B\rangle$. Let
$S_k$ for $k=0,1,\ldots,m$ be the largest subautomaton of $\A$ such that $S_k\setminus \B$ contains at most $k$ cycles
connected by a directed path.  Then every $S_k$ is also symmetric and self-similar. We pass to a power of the alphabet so
that every cycle in $\A\setminus \B$, and hence in $S_k\setminus S_0$, is actually a loop; and for every $s\in
S_k\setminus S_0$ and $x\in X$ either $s|_x\in S_{k-1}$ or $s|_x$ is circuit.

Let $G_k$ be the group generated by $S_k$. In particular, $G=G_m$ and we use notation $S=S_m$. The group $G_0$ is
contracting with nucleus $\nucl$. Indeed, by construction, there are no cycles in $S_0\setminus\nucl$, and hence
$s|_v\in\nucl$ for all $s\in S_0$ and all words $v$ of length greater than the diameter of~$S_0\setminus\nucl$.

\begin{lemma}\label{lemma_cycles_are_loops_generalized}
The length of cycles in the complete automaton $\A(G)$ of the group $G$ is bounded.
\end{lemma}
\begin{proof}
Take a circuit element $g\in G$. Let $g|_v=g$ for a nonempty word $v$, and $g|_u\neq g$ for every beginning $u$ of $v$.
Over all presentation of $g$ as a product $s_1s_2\ldots s_n$ for $s_i\in S$ with $n=l(g)$, we consequently choose
presentations with the maximal number of circuit elements $s_i\in S_k$ and then with the maximal number of any elements
$s_i\in S_k$ for $k=0,1,\ldots,m$. We obtain some specific presentations of the element $g$. Let $g=s_1s_2\ldots s_n$ be
one of such presentations and consider the induced presentation~\eqref{eqn_induced_present}:
\[
g=g|_v=s_1|_{v_1} s_2|_{v_2}\ldots s_n|_{v_n},
\]
where $v_{i-1}=s_i(v_i)$ with $v_n=v$. Notice, that this presentation also satisfies the above assumptions, as is the
induced presentation $g=g|_{v\ldots v}$ for every iteration $v\ldots v$ of the word $v$. In particular, every $s_i$ is
actually circuit, because otherwise the presentation given by $g=g|_{v\ldots v}$ would contain more circuit generators
than the chosen one.

Two consecutive elements $s_i$ and $s_{i+1}$ in the presentation cannot both lie in $\nucl$. Indeed, in this case
$(s_is_{i+1})|_{u}\in\nucl\subset S$ for all sufficiently large $u$. Hence $g=g|_{v\ldots v}$ can be expressed as a
product of less than $n$ generators for large enough iteration $v\ldots v$, contradicting $n=l(g)$.

Our choice of the presentation $g=s_1s_2\ldots s_n$ and the assumptions on the generating set $S$ force $s_i|_{v_i}=s_i$
for every $s_i\not\in\nucl$. Hence $v_i=x_i^m$ (here $m=|v|$), where $x_i\in X$ is the unique letter such that
$s_i|_{x_i}=s_i$. If $s_i\in\nucl$ for $i<n$ then $s_{i+1}\not\in\nucl$ and we still get $v_i=x_i^m$ with
$x_i=s_{i+1}(x_{i+1})$.

Assume $s_n\not\in\nucl$, and hence $v=x^m$ with $x=x_n$. Consider the induced presentations of $g|_{x^l}$ for every
$l=1,2,\ldots$. If $s_i\not\in\nucl$ then the $i$-th generator in the presentation of every $g|_{x^l}$ remains the same
and is equal to $s_i$. Let us trace the positions of elements from $\nucl$. If $s_i\in\nucl$ then the $i$-th element in
the presentation of $g|_{x^l}$ is equal to $s_i|_{x_i^l}$. Hence these $i$-th elements change according to the vertices
of the path in the nucleus $\nucl$ that starts at the state $s_i$ and goes along arrows with left label $x_i$. It follows
that every position is pre-periodic. We can eliminate the pre-periods by passing from the presentation $g=s_1s_2\ldots
s_n$ to the induced presentation $g=g|_{x^l}$ for large enough $l$, and hence we assume that every position repeats
periodically. If $l$ is a multiple of the lengths of every cycle in $\nucl$, then the induced presentation of $g|_{x^l}$
coincides with $s_1s_2\ldots s_n$ letter by letter, and hence $g=g|_{x^l}$. It follows that the length of the word $v$ is
bounded by the least common multiple of the length of cycles in the nucleus $\nucl$.

If $s_n\in\nucl$ then $s_{n-1}\not\in\nucl$ and $v=s_{n}^{-1}(x^m)$ with $x=x_{n-1}$. We can apply the same arguments as
above to the induced presentations of $g|_{s_n^{-1}(x^l)}$ for $l=1,2,\ldots$. The only difference is that the $n$-th
element in the induced presentations changes according to the path in the nucleus $\nucl$ that starts at the state $s_n$
and goes along arrows with right label~$x$.
\end{proof}

Let $L$ be the upper bound on the length of cycles in $\A(G)$. Note that in general we cannot pass to a power of the
alphabet to transform every cycle into a loop, i.e., to make $L=1$, as was possible in the case of polynomial automata.

Further, we will use the fact proved in Lemma~\ref{lemma_cycles_are_loops_generalized}, that if $g\in G_k$ for $k\geq 1$
is circuit and $g|_v=g$ then $g|_u\in G_{k-1}$ for every $u\in X^{|v|}$, $u\neq v$, and the length of $g|_u$ with respect
to $S_{k-1}$ is not greater than the length of $g$ with respect to~$S_k$.

The formulation of Lemma~1 remains the same (with $m\geq 1$).

\begin{lemma}\label{lemma_weak_contr_poly_generalized}
There exists a constant $C$ such that for every $g=s_1s_2\ldots s_n\in G$ for $s_i\in S$ and every word $v\in X^{*}$ of
length $|v|\geq C\left(\log n\right)^{m+1}$ the element $g|_v$ is either circuit, or $g|_v\in G_{m-1}$.
\end{lemma}
\begin{proof}
The proof is basically the same as the one of Lemma~1*. Let us indicate only the main argument.

We assume the lemma holds for the groups $G_{k}$ for $k=1,\ldots,m-1$ with some common constant $C_1$. Let $g\in G$ be a
circuit element. Consider the product $hg$ for $h\in G_{m-1}$ and its restrictions $(hg)|_v=h|_{g(v)} g|_{v}$ for words
$v\in X^{*}$ of length $|v|\geq C_1\left(\log l(h)\right)^m$. We have the following cases. If $g|_v\in G_{m-1}$ then
$(hg)|_v\in G_{m-1}$. So further assume $g|_v\not\in G_{m-1}$, and hence $g|_v$ and $g$ lie on the same cycle. There
exists a nonempty word $u$ of length $|u|\leq L$ such that $g|_v|_u=g|_v$. Let us look at $h|_{g(v)}$, which either
belongs to $G_{m-2}$ or is circuit. If $h|_{g(v)}\in G_{m-2}$ then $(hg)|_v=h'g'$ for circuit $g'=g|_v$ and
$h'=h|_{g(v)}\in G_{m-2}$. Otherwise $h|_{g(v)}$ is circuit, and let $l\leq L$ be the length of the cycle at $h|_{g(v)}$.
Then the length of the word $u^l$ is a multiple of $|u|$ and of $l$. If $h|_{g(v)}|_{g(u^l)}=h|_{g(v)}$ then
$(h|_{g(v)}g|_v)|_{u^l}=h|_{g(v)}g|_v$ and the state $(hg)|_v$ is circuit; otherwise $h|_{g(v)}|_{g(u^l)}\in G_{m-2}$.
Hence in all cases the state $(hg)|_{w}$ for a word $w$ of length $|w|\geq C_1\left(\log l(h)\right)^m+L^2$ is either
circuit, or it belongs to $G_{m-1}$, or it is of the form $h'g'$ for circuit $g'$ and $h'\in G_{m-2}$ with $l(h')\leq
l(h)$. Eventually, in the worst case, we will need to consider the product $hg$ for circuit $g$ and $h\in\nucl\subset
S_0$. In this case we can apply the same arguments as in the proof of Lemma~\ref{lemma_cycles_are_loops_generalized}.

The rest of the proof is the same.
\end{proof}

Now we can return to the growth of Schreier graphs. If $m=0$ then the group $G$ is contracting and the statement holds.
We suppose by induction that the statement holds for the group $G_{m-1}$ with some constant $C_1$.

Take a sequence $w=x_1x_2\ldots\in\xo$ and consider the Schreier graph $\gr_w=\gr_w(G,S)$. We will use the estimate
\eqref{eqn_theorem_growth_estimate}.

Let $H_n=\{h\in G_{m-1}: l_{S_{m-1}}(h)\leq n\}$ be the ball of radius $n$ in the group $G_{m-1}$ with respect to its
generating set $S_{m-1}$.

Let $k$ be the least integer greater than $C(\log n)^{m+1}$ from Lemma~\ref{lemma_weak_contr_poly_generalized}.  Then for
every $g\in\nucl_{(n,k)}\setminus H_n$ there exists a nonempty word $u$ of length $|u|\leq L$ such that $g|_u=g$. Hence
\[
\nucl_{(n,k)}\subset \bigcup_{|u|\leq L} \nucl^u_{(n,k)} \cup H_n,
\]
where $\nucl^u_{(n,k)}=\{g\in\nucl_{(n,k)} : g|_u=g\}$. By induction hypothesis the size of the orbit $H_n(v)$ is not
greater than $|X|^{C_1(\log n)^m}$. Let us estimate the size of the orbits $\nucl_{(n,k)}^u(v)$. If
$v=uu\ldots=u^{\infty}$ then for every $g\in\nucl_{(n,k)}^u$ we get $g(v)=z^{\infty}$ for some $z\in X^{|u|}$. If
$v=u^lu_1v_1$ for $u_1\in X^{|u|}$, $u_1\neq u$, and $l\geq 0$, then for every $g\in\nucl_{(n,k)}^u$ we get
$g(v)=z^lz_1h(v_1)$ for some $z,z_1\in X^{|u|}$ and $h\in H_n$. Summarizing all estimates we get
\[
|B(w,n)|\leq |X|^{C(\log n)^{m+1}+1}\cdot |X|^{1+2+\ldots+L}\cdot|X|^{2L}\cdot |X|^{C_1(\log n)^m}\leq |X|^{C_2(\log
n)^{m+1}}
\]
for some constant $C_2$.
\end{proof}

\section{Examples}\label{section_Examples}

\begin{figure}
\begin{center}
\psfrag{a}{$a$} \psfrag{id}{$e$} \psfrag{a1}{$a_1$} \psfrag{a2}{$a_2$} \psfrag{ak}{$a_k$} \psfrag{1|0}{$1|0$}
\psfrag{0|1}{$0|1$} \psfrag{0|0}{$0|0$} \psfrag{1|1}{$1|1$} \psfrag{x1}{$x_1|x_1$} \psfrag{x2}{$x_2|x_2$}
\psfrag{xk}{$x_k|x_k$} \psfrag{y1}{$\overline{x}_1|\overline{x}_1$} \psfrag{y2}{$\overline{x}_2|\overline{x}_2$}
\psfrag{y3}{$\overline{x}_3|\overline{x}_3$} \psfrag{yk}{$\overline{x}_k|\overline{x}_k$}
\epsfig{file=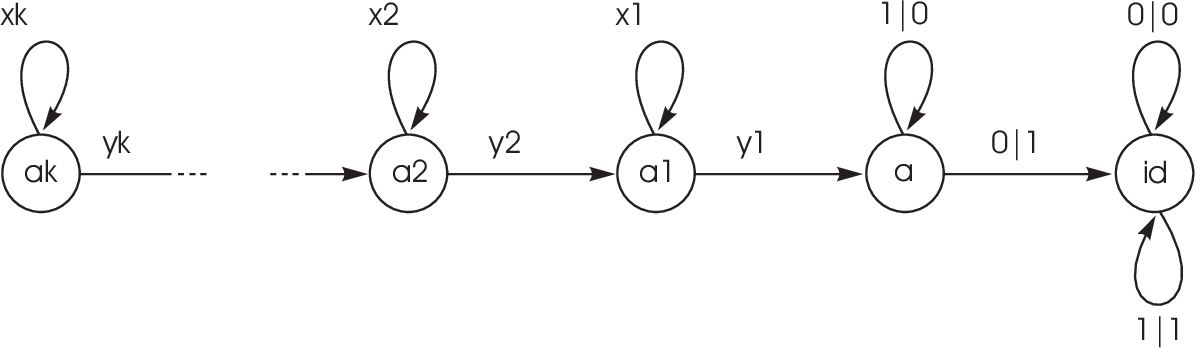,width=380pt} \caption{The automaton $\A_v$ for $v=x_1x_2\ldots x_k$}\label{fig_AutomatonOmega}
\end{center}
\end{figure}

\textbf{Omega-periodic graphs}. Let $X=\{0,1\}$ be the binary alphabet. For every finite word $v=x_1x_2\ldots x_k$ over
$X$ consider the automaton $\A_v$ shown in Figure~\ref{fig_AutomatonOmega}, where we use notation $\overline{0}=1$ and
$\overline{1}=0$. Every automaton $\A_v$ is polynomial of degree $k=|v|$. The automaton $\A_v$ is a subautomaton of
$\A_{vx}$, and we get an increasing chain of polynomial automata
\[
\A_{\emptyset}\subset\A_{x_1}\subset\A_{x_1x_2}\subset\ldots,
\]
for every $x_i\in X$. Hence every orbital Schreier graph $\gr_w(\A_v)$ for $w\in\xo$ is a subgraph of $\gr_w(\A_{vx})$.
This allows us to construct the Schreier graph $\gr_w(\A_v)$ consequently as $\gr_w(\A_{\emptyset})\subset
\gr_w(\A_{x_1})\subset \ldots\subset \gr_w(\A_{x_1x_2\ldots x_k})$ by looking at the action of the states $a,
a_1,a_2,\ldots,a_k$.

We start from the orbital Schreier graphs $\gr_w(\A_{\emptyset})$ of the automaton $\A_{\emptyset}$. The transformation
$a$ is called the adding machine, because its action on a sequence $y_1y_2\ldots\in\xo$ corresponds to the addition of
$1$ to the binary number $\sum_{i\geq 1} y_i2^{i-1}\in\mathbb{Z}_{2}$. In particular, the infinite cyclic group generated
by $\A_{\emptyset}$ acts faithfully on its every orbit on $\xo$. Hence every Schreier graph $\gr_w(\A_{\emptyset})$ for
$w\in\xo$ is a \textquotedblleft line\textquotedblright, i.e., its vertices can be identified with $\mathbb{Z}$ via the
map $a^m(w)\mapsto m$ and the edges are $(m,m+1)$ for all $m\in\mathbb{Z}$.

Notice that for every state $a_k$ of the automaton $\A_v$ if $a_k(y_1y_2\ldots)=z_1z_2\ldots$, for $y_i,z_i\in X$, and we
take the first position $n$ with $y_n\neq z_n$, then $a_k|_{y_1y_2\ldots y_{n-1}}=a$ and
$a(y_ny_{n+1}\ldots)=z_{n}z_{n+1}\ldots$. It follows that each $a_k$ preserves the orbits of the action of $a$ on $\xo$.
Hence the orbits of $(\A_v,\xo)$ for every word $v$ coincide with the orbits of $(\A_{\emptyset},\xo)$, and we can
identify the vertex set of every Schreier graph $\gr_w(\A_v)$ for $w\in\xo$ with the set $\mathbb{Z}$.

Every state $a_k$ of the automaton $\A_{x_1x_2\ldots x_k}$ acts on the infinite sequences as follows
\begin{equation}\label{eqn_exampl_OmegPeriod_ActionAk}
a_k(x_k^{n_k}\,\overline{x}_k \ldots x_2^{n_2}\,\overline{x}_2\,  x_1^{n_1}\,\overline{x}_1 w)=x_k^{n_k}\,\overline{x}_k
\ldots x_2^{n_2}\,\overline{x}_2\,  x_1^{n_1}\,\overline{x}_1 a(w)
\end{equation}
for every $w\in\xo$ and $n_i\in \mathbb{N}\cup \{0,\infty\}$. Notice that every sequence from $\xo$ appears in
\eqref{eqn_exampl_OmegPeriod_ActionAk} for suitable numbers $n_i$. Further, for every fixed $n_1,\ldots, n_k$, the edges
defined by \eqref{eqn_exampl_OmegPeriod_ActionAk}, when $w$ runs through $\xo$, are periodic, namely, these edges are
invariant under the shift $m\mapsto m+2^{n_1+\ldots +n_k+k}$ on $\mathbb{Z}$. It follows that all Schreier graphs
$\gr_w(\A_v)$ for $w\in\xo$ are unions of periodic subgraphs on $\mathbb{Z}$, and hence they are $\omega$-periodic in the
terminology of \cite{omega_periodic}.

\begin{figure}
\begin{center}
\epsfig{file=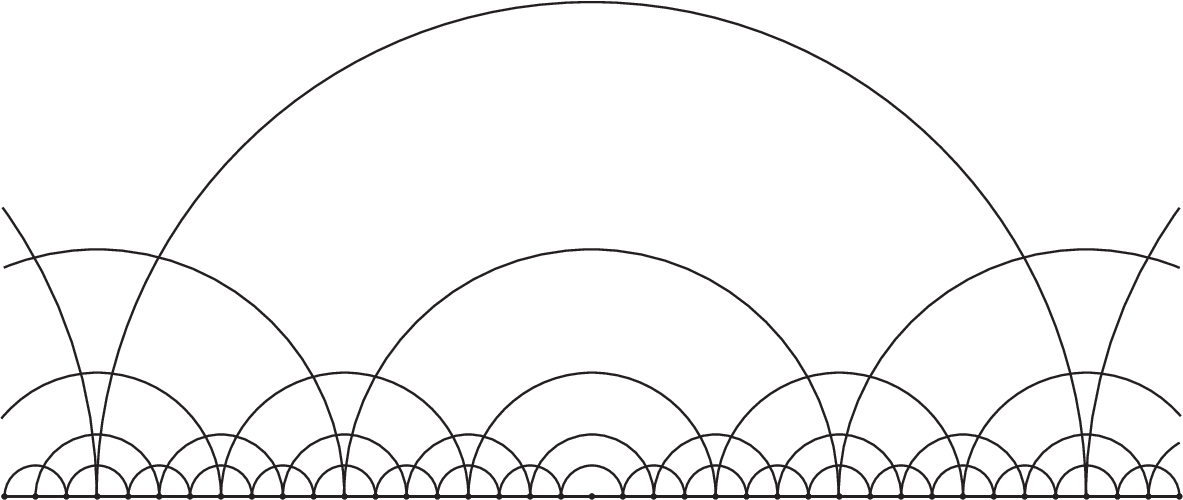,height=180pt} \caption{The Schreier graph $\gr_{0^{\infty}}(\A_0)$}\label{fig_OmegaGraph}
\end{center}
\end{figure}

The Schreier graph $\gr_{0^{\infty}}(\A_0)$ (see Figure~\ref{fig_OmegaGraph}) was defined in \cite{omega_periodic} as an
example of a graph of intermediate growth (its growth is equivalent to $n^{\log_4 n}$) connected to the long range
percolation on $\mathbb{Z}$. All orbital Schreier graphs $\gr_w(\A_0)$ for $w\in\xo$ of the automaton $\A_0$ were studied
in \cite{SchreierInter}, where it was proved that the family of these graphs contains uncountably many pairwise
nonisomorphic graphs, all of them except $\gr_{0^{\infty}}$ are locally isomorphic, and all have intermediate growth. The
groups generated by $\A_0$ and $\A_1$ are the same, and the automata $\A_0, \A_1$ are just different generating sets of
this group. Hence the Schreier graphs $\gr_w(\A_1)$ also have intermediate growth. The Schreier graphs $\gr_w(\A_v)$ for
every nonempty word $v$ have growth not less than the growth of $\gr_w(\A_0)$ or of $\gr_w(\A_1)$. Hence their growth is
superpolynomial, and then Theorem~\ref{theorem_main} implies that all orbital Schreier graphs $\gr_w(\A_v)$ for $w\in\xo$
and $v\neq\emptyset$ have intermediate growth. Let us make more precise estimates to show that the upper bound in
Theorem~\ref{theorem_main} is sharp for groups generated by polynomial automata of degree $m$.

\begin{theorem}
For every $m\in\mathbb{N}$ there exist constants $A\geq B>1$ such that all orbital Schreier graphs $\gr_w(\A_{0^m})$ for
$w\in X^{\omega}$ have intermediate growth between $A^{(\log n)^{m+1}}$ and $B^{(\log n)^{m+1}}$.
\end{theorem}
\begin{proof}
The upper bound follows from Theorem~\ref{theorem_main}. Let us prove the lower bound.

Consider the graph $\gr_{0^{\infty}}(\A_{0^m})$ and its subgraph $\gr(m,n)$ induced by the set of vertices
$X^n0^{\infty}$. Notice that the graph $\gr(m,n)$ embeds in every Schreier graph $\gr_w(\A_{0^m})$ for $w\in X^{\omega}$.
Indeed, if $a_i(v0^{\infty})=u0^{\infty}$ for different words $v,u\in X^n$ then $a_i|_{v}=1$ and $a_i(vw_1)=uw_1$ for
every $w_1\in X^{\omega}$. Hence, if $w=v_1w_1$ for $v_1\in X^n$ and $w_1\in X^{\omega}$ then the map $v0^{\infty}\mapsto
vw_1$ is an embedding of the graph $\gr(m,n)$ into the graph $\gr_w(\A_{0^m})$. It follows that the ball $B(w,r)$ of
radius $r=\diam \gr(m,n)$ in every orbital Schreier graph $\gr_w(\A_{0^m})$ contains at least $2^n$ vertices, where
$\diam \gr$ is the diameter of the graph $\gr$. If we prove that $\diam \gr(m,n)\leq C^{n^{\frac{1}{m+1}}}$ for all
$n\geq 1$ with some constant $C$ depending only on $m$, then we get the necessary lower bound on the growth of every
Schreier graph $\gr_w(\A_{0^m})$ for $w\in X^{\omega}$. Equivalently, one can consider the approximation of orbital
Schreier graphs $\gr_w(\A_{0^m})$ by finite Schreier graphs $\gr_n(\A_{0^m})$ and prove that the diameter of
$\gr_n(\A_{0^m})$ is not greater than $C^{n^{\frac{1}{m+1}}}$ for some constant $C=C(m)$.

Let $d_m(n)=\diam \gr(m,n)$. For some $k\in\{1,2,\ldots,n-1\}$ to be chosen later consider the subgraph $\Lambda$ of
$\gr(m,n)$ induced by the set of vertices $0^{k-1}1X^{n-k}0^{\infty}$. Notice that
$a_i(0^{k-1}1v0^{\infty})=0^{k-1}1u0^{\infty}$ for different words $v,u\in X^{n-k}$ if and only if
$a_{i-1}(v0^{\infty})=u0^{\infty}$. Hence the map $0^{k-1}1v0^{\infty}\mapsto v0^{\infty}$ is an isomorphism between the
graph $\Lambda$ and the graph $\gr(m-1,n-k)$. Then the diameter of $\Lambda$ is equal to $d_{m-1}(n-k)$. Using periodic
structure of the graph $\gr(n,m)$ it is easy to see that the ball in $\gr(m,n)$ of radius $r=d_m(k)$ around every vertex
$0^{k-1}1v0^{\infty}$ for $v\in X^{n-k}$ contains all vertices from the set $X^kv0^{\infty}$. Hence the union of these
balls cover the whole graph $\gr(m,n)$. It follows that
\begin{equation}\label{eqn_inequality_dm(n)}
d_m(n)\leq 2d_m(k) + d_{m-1}(n-k) %\ \mbox{ for every } \ k=1,2,\ldots, n-1.
\end{equation}
Based only on this inequality and the initial conditions $d_0(n)=2^{n}-1\leq 2^n$, $d_m(1)=1$, $n,m\geq 1$, we will make
an upper bound on $d_m(n)$ using the same arguments as in \cite{omega_periodic}.

For every $m\geq 0 $ let us define the sequence $\{y_m(t)\}_{t\in\mathbb{N}}$ recursively as follows. Put $y_0(t)=t$ for
all $t\geq 1$ and for every $m\geq 1$ define
\begin{eqnarray*}
y_m(1)=1 \ \mbox{ and } \  y_m(t)=y_m(t-1)+y_{m-1}(t-1), t\geq 2.
\end{eqnarray*}
In order to estimate the values of $y_m(t)$ and $d_m(y_m(t))$ we will use the following basic inequality
\[
\frac{1}{m+1}t^{m+1}\leq 1^m+2^m+\ldots+t^m\leq (t+1)^{m+1}.
\]
Assume inductively that $at^m\leq y_{m-1}(t)\leq t^m$ for all $t\geq 1$ with some constant $a>0$ (depending only on $m$).
Then
\begin{eqnarray*}
y_m(t)&\leq& 1^m+\ldots+(t-2)^m+(t-1)^m\leq t^{m+1},\\
y_m(t)&\geq& a(1^m+\ldots+(t-2)^m+(t-1)^m)\geq \frac{a}{m+1}(t-1)^{m+1}.
\end{eqnarray*}
It follows that there exists a constant $b=b(m)>0$ such that $bt^{m+1}\leq y_m(t)\leq t^{m+1}$ for all $t\geq 1$.

To estimate $d_m(y_m(t))$ we use Equation~(\ref{eqn_inequality_dm(n)}) with $k=y_m(t-1)$. By definition
$d_0(y_0(t))=d_0(t)\leq 2^t$ and let us suppose by induction on $m$ that $d_m(y_m(t))\leq t^{m}2^{t}$ for all $t\geq 1$.
Then
\begin{eqnarray*}
d_{m+1}(y_{m+1}(t))&\leq& 2d_{m+1}(y_{m+1}(t-1)) + d_m(y_m(t-1))\leq\\
&\leq& 2d_{m+1}(y_{m+1}(t-1))+(t-1)^m2^{t-1}\leq\\
&\leq& 2d_{m+1}(y_{m+1}(t-2))+(t-2)^m2^{t-1}+(t-1)^m2^{t-1}\leq\\
&\leq& 2^{t-1}(1^m+2^m+\ldots +(t-1)^m)\leq t^{m+1}2^{t-1}\leq t^{m+1}2^{t}
\end{eqnarray*}
for all $t\geq 1$. The estimates on the sequence $y_m(t)$ and the upper bound on $d_m(y_m(t))$ imply that there exists a
constant $C=C(m)$ such that $d_m(n)\leq C^{n^{\frac{1}{m+1}}}$ for all $n\geq 1$.
\end{proof}

It is worse to notice that if we take the automaton shown in Figure~\ref{fig_AutomatonOmega} and change the labels of the
edges passing from the states $a_1,\ldots,a_k$ in any way we want, we still get an automaton whose all Schreier graphs
$\gr_w$ have intermediate growth. More generally, if we take any polynomial automaton $\A$ over $X$ that contains
$\A_{0}$ (with any labels at $a_1$) as a subautomaton, then all orbital Schreier graphs $\gr_w(\A)$ have intermediate
growth.

\vspace{0.2cm} \noindent\textbf{Hanoi graphs}. The Hanoi Tower Game on $k$ pegs is played with $n$ disks of distinct size
placed on $k$ pegs, $k\geq 3$. Initially, all disks are placed on the first peg according to their size so that the
smallest disk is at the top, and the largest disk is at the bottom. A player can move only one top disk at a time from
one peg to another peg, and can never place a bigger disk over a smaller disk. The goal of the game is to transfer the
disks from the first peg to another peg. More generally, one can take any two configurations of disks and ask a player to
transform one configuration to another one, where a configuration is any arrangement of $n$ disks on $k$ pegs such that a
bigger disk does not lie on a smaller disk. For more information about this game, its history, solutions, and open
problems, we refer the reader to \cite{hinz:hanoi,frame-stewart,FrStConjecture,gri_sunik:hanoi} and the references
therein.

The Hanoi Tower Game is modeled by the Hanoi graphs $H_n^{(k)}$. The vertices of the graph $H_n^{(k)}$ are the
configurations of $n$ disks on $k$ pegs, and two configurations are adjacent if one can pass from one to another by a
single disk move. Then the Hanoi Tower Game is equivalent to finding a path in the graph $H_n^{(k)}$ between two given
vertices. In particular, the diameter of the Hanoi graph gives an upper bound on the optimal number of steps in the Hanoi
Tower Game.

\begin{figure}
\begin{center}
\psfrag{a12}{$a_{(12)}$} \psfrag{a13}{$a_{(13)}$} \psfrag{a14}{$a_{(14)}$} \psfrag{a23}{$a_{(23)}$}
\psfrag{a24}{$a_{(24)}$} \psfrag{a34}{$a_{(34)}$} \psfrag{id}{$e$} \psfrag{1|1}{$1|1$} \psfrag{2|2}{$2|2$}
\psfrag{3|3}{$3|3$} \psfrag{4|4}{$4|4$} \psfrag{1|2}{$1|2$} \psfrag{2|1}{$2|1$} \psfrag{1|3}{$1|3$} \psfrag{3|1}{$3|1$}
\psfrag{1|4}{$1|4$} \psfrag{4|1}{$4|1$} \psfrag{2|3}{$2|3$} \psfrag{3|2}{$3|2$} \psfrag{2|4}{$2|4$} \psfrag{4|2}{$4|2$}
\psfrag{3|4}{$3|4$} \psfrag{4|3}{$4|3$}\epsfig{file=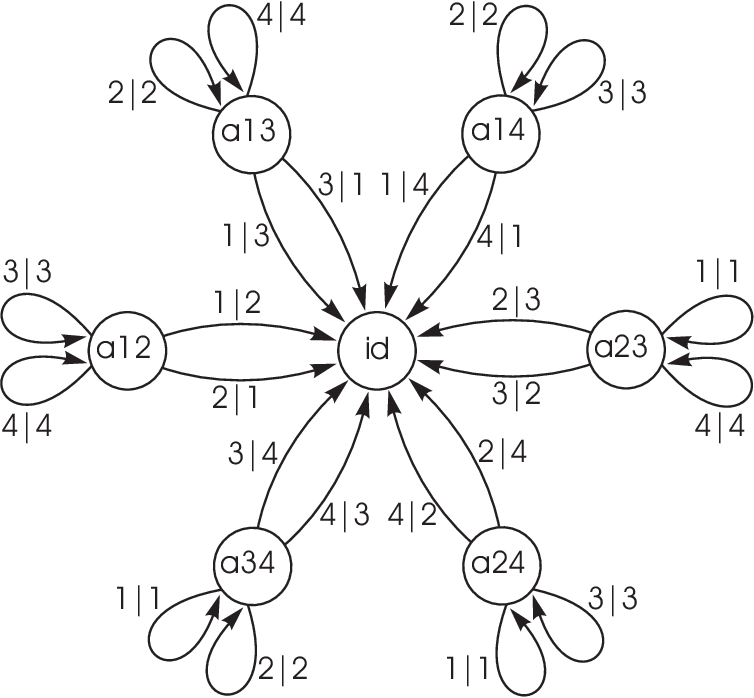,width=260pt} \caption{The Hanoi Towers automaton
$\A_{(4)}$ on $4$ pegs}\label{fig_AutomHanoi}
\end{center}
\end{figure}

It was noticed in \cite{gri_sunik:hanoi} that the game can be also modeled by a finite automaton $\A_{(k)}$ over the
alphabet $X=\{1,2,\ldots,k\}$, called the Hanoi Towers automaton on $k$ pegs. In this model, the pegs are identified with
the letters of the alphabet $X$, and the disks are labeled by $1,2,\ldots, n$ according to their size. Then the
configurations of $n$ disks on $k$ pegs can be encoded by words of length $n$ over $X$, where the word $x_1x_2\ldots x_n$
corresponds to the unique configuration in which the disk $i$ is placed on the peg $x_i$. The automaton $\A_{(k)}$ has
the trivial state $e$ and the state $a_{(ij)}$ for every transposition $(i,j)$ on $X$ with two arrows
$a_{(ij)}\rightarrow e$ labeled by $i|j$ and $j|i$, and the loop at $a_{(ij)}$ labeled by $x|x$ for every $x\in
X\setminus\{i,j\}$. For example, the automaton $\A_{(4)}$ is shown in Figure~\ref{fig_AutomHanoi} (the loops at the
trivial state $e$ are not drawn). The action of the state $a_{(ij)}$ on a word of length $n$ corresponds to a single disk
move between the pegs $i$ and $j$. Hence the Schreier graph $\gr_n(\A_{(k)})$ is precisely the Hanoi graph $H_n^{(k)}$.

The complexity of the Hanoi Tower Game highly depends on whether $k=3$ or $k\geq 4$. In the case $k=3$ the Hanoi graph
$H_n^{(3)}$ has diameter $2^n-1$, and this is the smallest number of steps to win the game. The automaton $\A_{(3)}$ is
bounded and generates a contracting group. All orbital Schreier graphs $\gr_w(\A_{(3)})$ for $w\in\xo$ have polynomial
growth of degree $\frac{\log 3}{\log 2}$.

For $k\geq 4$ the Hanoi Tower Game has subexponential complexity. Namely, it can be
solved in $2^{(1\pm o(1))(n(k-2)!)^{\frac{1}{k-2}}}$ moves using the Frame-Stewart
algorithm (see \cite{frame-stewart}) and
 this is asymptotically the smallest number of moves to win the game (see \cite{FrStConjecture}). In particular, the diameter of the Hanoi graph
$H_n^{(k)}$ and the Schreier graph $\gr_n(\A_{(k)})$ is asymptotically $\exp(n^{\frac{1}{k-2}})$. One can apply this
asymptotic estimate to the orbital Schreier graphs $\gr_w(\A_{(k)})$ for $w\in\xo$. It is easy to see that the growth
function $\gamma$ of each graph $\gr_w(\A_{(k)})$ satisfies $\gamma(d_n)\geq k^n$, where $d_n$ is the diameter of the
graph $\gr_n(\A_{(k)})$. Hence the upper bound on the diameter of $\gr_n(\A_{(k)})$ implies a superpolynomial lower bound
on the growth of $\gr_w(\A_{(k)})$, namely $\gamma(m)\geq k^{C(\log m)^{k-2}}$ with some constant $C>0$. This was used in
\cite[Theorem~2.1]{gri_sunik:hanoi} to conclude that the Schreier graph $\gr_{0^{\infty}}(\A_{(k)})$ for $k\geq 4$ has
intermediate growth between $a^{(\log m)^{k-2}}$ and $b^{(\log m)^{k-2}}$ for some constants $b>a>1$.

The automaton $\A_{(k)}$ for $k\geq 4$ is not polynomial, and does not satisfy the assumptions of
Theorem~\ref{theorem_growth_Schr_general}. However, we can apply similar ideas to give the subexponential upper bound
$k^{C(\log m)^{k-2}}$ on the growth of every Schreier graph $\gr_w(\A_{(k)})$ for $w\in\xo$.

\begin{theorem}
All orbital Schreier graphs $\gr_w(\A_{(k)})$ for $k\geq 4$ and $w\in\xo$ have intermediate growth.
\end{theorem}
\begin{proof}
To provide an upper bound we follow similar arguments as in the proofs on Theorems \ref{theorem_growth_Schreier_poly} and
\ref{theorem_growth_Schr_general}. We will show only the case $k=4$, the general case is analogous.

Let $G_{(k)}$ be the automaton group generated by $\A_{(k)}$. Notice that for every state $s$ of the automaton $\A_{(k)}$
and every word $v\in X^{*}$ if $s|_v\neq 1$ then $s(v)=v$ and $s|_v=s$. It immediately follows that every cycle in the
complete automaton $\A(G_{(k)})$ of the group $G_{(k)}$ is a loop labeled by $x|x$ for some letter $x\in X$.

The group $G_{(3)}$ is contracting with nucleus $\nucl=\A_{(3)}$. Then for every $g=s_1s_2\ldots s_n\in G_{(3)}$,
$s_i\in\A_{(3)}$, the restriction $g|_v$ belongs to $\nucl$ for all words $v$ of length $|v|\geq \log n$. Hence every
Schreier graph $\gr_w(\A_{(3)})$ has subexponential (polynomial) growth not greater than $4|X|^{\log n+1}\leq
|X|^{C_1\log n}$ for some constant $C_1$.

Now consider the group $G_{(4)}$ and the automaton $\A_{(4)}$. For $x\in X$ denote by $\A_{(4)}^x$ the subautomaton of
$\A_{(4)}$ consisting of states that fix letter $x$. The group generated by $\A_{(4)}^x$ acts on the words over
$X\setminus\{x\}$ in the same way as the group $G_{(3)}$ acts on $\{1,2,3\}^{*}$, and that is where we can apply the
inductive arguments. In particular, if $g=s_1s_2\ldots s_n$ for $s_i\in\A_{(4)}^x$ then for every word $v\in X^{*}$ that
contains at least $\log n$ letters different from $x$ the restriction $g|_v$ belongs to $\A_{(4)}$, in other words,
$g|_v$ fixes one more letter except $x$.

Take two different letters $x,y\in X$ and consider elements $g=s_1s_2\ldots s_k$ for $s_i\in \A_{(4)}^x$ and
$h=t_1t_2\ldots t_m$ for $t_i\in\A_{(4)}^y$. Take an arbitrary word $v\in X^{*}$ of length $|v|\geq 2 \log n$ with
$n=k+m$. If the word $v$ contains at least $\log n$ letters not equal to $x$ then $g|_v\in\A_{(4)}$. Otherwise, $v$
contains at least $\log n$ letters equal to $x$, and then $g(v)$ contains at least $\log n$ letters equal to $x$. Since
$x\neq y$, we get $h|_{g(v)}\in\A_{(4)}$. It follows that for every word of length $|v|\geq 2 \log n+1$ there exists a
letter $z\in X$ such that every generator in the induced presentation of $(hg)|_v=h|_{g(v)}g|_v$ belongs to $\A_{(4)}^z$
(i.e., $(hg)|_v$ is circuit).

Let us prove that there exists a constant $C_2$ such that for every product $g=s_1s_2\ldots s_n\in G_{(4)}$,
$s_i\in\A_{(4)}$, and every word $v\in X^{*}$ of length $|v|\geq C_2 (\log n)^2$ there exists $x\in X$ such that every
generator $s_i|_{v_i}$ in the induced presentation $g|_v=s_1|_{v_1} s_2|_{v_2}\ldots s_n|_{v_n}$ belongs to $\A_{(4)}^x$
(hence $g|_v$ is circuit, here $g|_v|_x=g|_v$). We partition the presentation $g=s_1s_2\ldots s_n$ on blocks
\begin{equation}\label{eqn_exam_Hanoi_g=g1g2...}
g=g_1g_2\ldots g_l
\end{equation}
such that every $g_i=s_{j_i}s_{j_i+1}\ldots s_{j_{i+1}-1}$ contains only generators from $\A_{(4)}^{x}$ for some letter
$x\in X$ (depending on $i$). Consider the products $g_1g_2$, $g_3g_4$, \ldots, and their restrictions on words $v$ of
length $\geq 2\log n+1$. Using the above property we get that the element $g|_v$ can be expressed as a product
\eqref{eqn_exam_Hanoi_g=g1g2...} with $\leq (l+1)/2$ elements $g_i$. Applying the same procedure $\log l$ times, we get
an element with needed properties. The existence of the constant $C_2$ follows.

Consider the Schreier graph $\gr_{w}(\A_{(4)})$ for $w=x_1x_2\ldots\in\xo$. We will use estimate
\eqref{eqn_theorem_growth_estimate} with $k$ being the least integer greater than $C_2(\log n)^2$. Then
\[
\nucl_{(n,k)}=\{g|_{x_1x_2\ldots x_k} : g\in G_{(4)} \mbox{ and } l(g)\leq n\}\subset \bigcup_{x\in X} \nucl_{(n,k)}^x,
\]
where $\nucl_{(n,k)}^x$ consists of those elements $g\in\nucl_{(n,k)}$ that can be expressed as a product of no more than
$n$ elements from $\A_{(4)}^x$. Every element $g\in\nucl_{(n,k)}^x$ fixes every occurrence of the letter $x$ in the
sequence $x_{k+1}x_{k+2}\ldots$, and changes every other letter in the same way as the group $G_{(3)}$ acts on
$X\setminus\{x\}$. Hence the orbit $\nucl_{(n,k)}^x(x_{k+1}x_{k+2}\ldots)$ has subexponential growth not greater than
$(|X|-1)^{C_1\log n}$. We get the upper bound
\[
|B(w,n)|\leq |X|^{C_2(\log n)^2+1}\cdot |X|\cdot (|X|-1)^{C_1 \log n} \leq |X|^{C(\log n)^2}
\]
for some constant $C$.
\end{proof}

\vspace{0.2cm} \noindent\textbf{Weakly contracting groups.} One of the main properties used in the proof of
Theorems~\ref{theorem_main} and \ref{theorem_growth_Schr_general} leads us to the following definition. We say that a
group $G$ generated by a finite automaton is \textit{weakly contracting} if the length of cycles in the complete
automaton $\A(G)$ of the group is bounded. Every finitely generated contracting group is weakly contracting. The groups
generated by polynomial automata, the groups from Theorem~\ref{theorem_growth_Schr_general}, and the Hanoi Towers groups
are also weakly contracting. It is natural to ask which properties can be generalized to this class of groups. Is the
word problem in weakly contracting groups solvable in subexponential time? Do not weakly contracting groups contain
non-abelian free subgroups? Do the orbital Schreier graphs $\gr_w$ of weakly contracting groups have subexponential
growth?

\end{document}